\documentclass{amsart}
\usepackage[utf8]{inputenc}
\usepackage{amsfonts, amsmath, amsthm, amssymb}
\usepackage{graphicx,color}

\title[New characterizations of operator monotone functions]{New characterizations of operator monotone functions}

\author[T. H. Dinh]{Trung Hoa Dinh}
\address{Division of Computational Mathematics and Engineering, Institute for Computational Science, Ton Duc Thang University, Ho Chi Minh City, Vietnam; \\ Faculty of Civil Engineering, Ton Duc Thang University, Ho Chi Minh City, Vietnam\\
Department of Mathematics and Statistics \\University of North Florida\\ Jacksonville, FL 32224, USA}
\email{dinhtrunghoa@tdt.edu.vn}
\author[R. Dumitru]{Raluca Dumitru}
\address{Department of Mathematics and Statistics \\ University of North Florida  \\ Jacksonville, FL 32224}
\email{raluca.dumitru@unf.edu}

\author[J. A. Franco]{Jose A. Franco}
\address{Department of Mathematics and Statistics \\ University of North Florida  \\ Jacksonville, FL 32224}
\email{jose.franco@unf.edu}


\theoremstyle{plain}
\newtheorem{theorem}{Theorem}
\newtheorem{lemma}[theorem]{Lemma}

\newtheorem{remark}[theorem]{Remark}

\newtheorem{proposition}[theorem]{Proposition}

\theoremstyle{definition}
\newtheorem{definition}{Definition}[section]

\begin{document}
\begin{abstract}
If $\sigma$ is a symmetric mean and $f$ is an operator monotone function on $[0, \infty)$, then 
$$f(2(A^{-1}+B^{-1})^{-1})\le f(A\sigma B)\le f((A+B)/2).$$
Conversely, Ando and Hiai showed that if $f$ is  a function that satisfies either one of these inequalities for all positive operators $A$ and $B$ and a symmetric mean different than the arithmetic and the harmonic mean, then the function is operator monotone. 

In this paper, we show that the arithmetic and the harmonic means can be replaced by the geometric mean to obtain similar characterizations. Moreover, we give characterizations of operator monotone functions using self-adjoint means and general means subject to a constraint due to Kubo and Ando.
\end{abstract}

\maketitle

\section{Introduction}
It is well-known that if $\sigma$ is a symmetric mean of operators,  {\it i.e.},  $A\sigma B =B\sigma A$, the following inequality is satisfied for any positive operators $A$ and $B$,
\begin{equation}\label{general}
A ! B \le A\sigma B \le A \nabla B,
\end{equation}
where $A ! B=2(A^{-1}+B^{-1})^{-1}$ is the harmonic mean of $A$ and $B$, and $A \nabla B=(A+B)/2$ is the arithmetic mean of $A$ and $B$. Obviously, if $f:[0,\infty)\to [0,\infty)$ is operator monotone, we have
\begin{equation}\label{general1}
f(A ! B) \le f(A\sigma B) \le f(A \nabla B).
\end{equation}
Interestingly, if a continuous function $f$ satisfies either of the inequalities for some scalar mean $M$,
\begin{equation}\label{general2}
f(a ! b) \le f(M(a, b)) \le f(a \nabla b).
\end{equation}
for positive numbers $a$ and $b$, then $f$ is monotone increasing. Matrix generalizations of this observation for Kubo-Ando means were discussed by Hiai and Ando in \cite[Proposition 4.1]{hiai-ando}. Namely, they showed that a continuous function $f$ on $(0, \infty)$ is operator monotone if and only if one of the following conditions holds:
\begin{itemize}
    \item[(A)] $f(A\nabla B) \ge f(A\sigma B)$ for all positive definite matrices $A, B$ and for some symmetric operator mean $\sigma \neq \nabla$;
    \item[(B)] $f(A! B) \le f(A\sigma B)$ for all positive definite matrices $A, B$ and for some symmetric operator mean $\sigma \neq !$.
\end{itemize}
 
Due to the importance of the geometric mean,
$$A\#B:= A^{1/2}(A^{-1/2}BA^{-1/2})^{1/2}A^{1/2},$$
as the mid-point of the geodesic,
$$A\#_tB:= A^{1/2}(A^{-1/2}BA^{-1/2})^{t}A^{1/2},$$
connecting two matrices $A$ and $B$ in the Riemannian manifold of positive matrices, it is natural to consider a similar characterization using this mid-point. This importance becomes more evident when one considers that $\#$ is not only symmetric but also self-adjoint {\it i.e.} $(A\# B)^{-1}=A^{-1}\# B^{-1}$, so it seems as a natural candidate to extend this characterization to other classes of means.  In this article we consider the following question:

{\bf Question:} Is the operator monotonicity of a continuous function $f$ on $(0, \infty)$ equivalent to one of the following conditions:
\begin{itemize}
\item[(C)] \label{Q3} $f(A\sharp B) \le f(A\sigma B)$ for all positive definite matrices $A, B$ and for some symmetric operator mean $\sigma > \sharp$;
\item[(D)] \label{Q4} $f(A\sigma B) \le f(A\tau B)$ for all positive matrices $A, B$ and for different operator means $\sigma$ and $\tau$ such that $\sigma < \tau$. 
\end{itemize}

In this article we answer (C) positively for any symmetric operator mean $\sigma\ne \#$ (see Proposition \ref{Prop14}). Regarding (D), we answer the question positively for two cases. The first case is when $\sigma=\#$ and $\tau$ is self-adjoint and satisfies a special order relation due to Hansen and Audenaert (see Theorem \ref{Prop11}). The second case is when $\sigma =\#$ and $\tau$ is any mean that satisfies a condition introduced by Kubo and Ando (see Theorem \ref{Prop12}).

The paper is organized as follows. In Section \ref{Sect2} we motivate the main results of the article by analyzing concrete means in the scalar case. In Section \ref{section 3} we establish matrix generalizations of the results of Section \ref{Sect2}. That is, we obtain characterizations of operator monotone functions by inequalities involving  the geometric mean and general operator means.  In Subsection \ref{Hansen}, we use a characterization of symmetric means due to Audenaert, Cai, and Hansen \cite{Audenaert} to give a partial answer to (C). This result is further generalized in the last subsection of this section where in Theorem \ref{Prop14} the question is answered definitively. In Subsection \ref{SelfAdjointSection}, we use a different characterization due to Hansen \cite{Hansen} for self-adjoint means to answer (D) partially for self-adjoint means. Subsection \ref{KAConditionSect} uses a condition on means introduced by Kubo and Ando to answer (D) on the class of means that satisfy this condition. The last section dedicates to characterizations of operator monotonicity by opeartor means interpolating between the geometric and the arithmetic means such as the matrix Heron means, the Heinz means. In particular, the result in Theorem \ref{theorem 5} answers question (D) positively in the case when $\sigma$ is the Heinz mean and $\tau$ is the Heron mean.

\section{Scalar means and characterization of monotone functions}\label{Sect2}

For two non-negative numbers $x$ and $y$ let us denote by $$G_s(x,y)=\frac{x^s y^{1-s}+x^{1-s}y^s}{2}$$ the Heinz means and by 
$$H_s(x,y)=s\frac{x+y}{2} + (1-s)x^{1/2}y^{1/2}$$ the Heron means. 

The family of Heron means and Heinz means are clearly interpolations between the arithmetic and the geometric means. In \cite{bhatia3}, Bhatia obtained a relation between the Heinz mean and the Heron mean which states that for $t \in [0, 1]$,  
\begin{equation}\label{bh1}
G_t(a, b) \le H_{(2t-1)^2}(a, b). 
\end{equation}
Therefore, for any $t\in [0, 1]$, we have 
\begin{equation}\label{general11}
\sqrt{ab} \le G_t(a, b) \le H_{(2t-1)^2} \le H_{|2t-1|}  \le \frac{a+b}{2}. 
\end{equation}

As mentioned in the introduction,  if a continuous function $f$ satisfies \eqref{general2}, the function is increasingly monotone. To motivate the answer to question (C), we prove a new characterization of monotonicity (as a real function) based on Inequality \eqref{general11}. 

\begin{theorem}\label{ScalarCase}
Let $M$ be some symmetric scalar mean on $\mathbb{R}^+$ which is strictly greater than to the geometric mean. And let 
\begin{equation}\label{geom}
f(\sqrt{ab}) \le f(M(a, b)) 
\end{equation}
whenever non-negative numbers $a$ and $b$. Then the function $f$ is increasingly monotone on $\mathbb{R}^+$.
\end{theorem}
\begin{proof}
To prove the theorem, we have to show that for any $0 < x \le y$ there exist $a, b>0$ such that $x= \sqrt{ab}$ and $y= M(a,b) = ah(b/a)$, where $h(t) = M(1, t)$ is the representing function of $M$.
Or, equivalently, for any $y_0\ge 1$ there exist $a, b>0$ such that $1= \sqrt{ab}$ and $y_0= M(a,b) = a^{-1}h(a^2)$ (because of the first identity). The function $\varphi(t)= t^{-1}h(t^2)$ is surjective from $(0, \infty)$ onto $[1, \gamma)$, where $\gamma=\lim_{t \to \infty}(t)>1$.  Therefore, for any $1 \le y_0 < \gamma $ there exists $a>0$ such that $y_0=a^{-1}h(a^2)$. Consequently, if $0 < x \le y \le \gamma x$, the previous argument implies that 
$$f(x)\le f(y).$$

If $y> \gamma x$, equivalently if $y_0>\gamma$, let $\gamma_0\in (1,\gamma)$ and consider the sequence $\{\gamma_0^n\}_{n\in \mathbb{N}}$. Since $\gamma_0^n \to \infty$ as $n\to \infty$, there exists $k\in\mathbb{N}$ such that $$0<x < \gamma_0 x\le ...\le \gamma_0^k x \le y < \gamma_0^{k+1} x.$$
Hence, the previous argument implies that:
$$f(x) < f(\gamma_0 x) \le ...\le f(\gamma_0^k x) \le f(y).$$
Therefore, $f$ is increasingly monotone on $\mathbb{R}^+$.
\end{proof}

Now, we show that the inequality between the Heinz mean and the Heron mean of scalars also characterizes monotonicity.
\begin{theorem}\label{Prop4} A continuous function $f$ on $[0, \infty)$ is monotone increasing if and only if for any pair of positive numbers $x, y$ and $s\in (0, 1/2)\cup (1/2,1)$, 
\begin{equation}\label{CharacterizationBhatia}
 f\left(\frac{x^s y ^{1-s}+x^{1-s}y^s}{2}\right) \le f\left(\alpha(s)^2\frac{x+y}{2} + (1-\alpha(s)^2)\sqrt{xy} \right), 
\end{equation}
where $\alpha(s)=2s-1$. 
\end{theorem}
\begin{proof}
The implication follows from \eqref{general11} and monotonicity, so we only need to show the converse. Given two positive numbers $a\le b$, it suffices to show that there exist positive numbers $x$ and $y$ such that
\begin{equation}\label{Prop4ab}
a =\frac{x^s y ^{1-s}+x^{1-s}y^s}{2}, \quad b =\alpha(s)^2\frac{x+y}{2} + (1-\alpha(s)^2)\sqrt{xy},
\end{equation}
as this would imply $f(a)\le f(b)$ showing the desired monotonicity.
If such $x$ and $y$ exist, from (\ref{Prop4ab}) we would have 
\begin{align*}
    \frac{a}{b}& =\frac{x^s y ^{1-s}+x^{1-s}y^s}{\alpha(s)^2(x+y) + 2(1-\alpha(s)^2)\sqrt{xy}} \\
    & = \frac{(y/x)^{\alpha(s)/2}+(y/x)^{-\alpha(s)/2}}{\alpha(s)^2((y/x)^{1/2}+(y/x)^{-1/2}) + 2(1-\alpha(s)^2)} \\
    & =\frac{\cosh(\alpha(s)c)}{\alpha(s)^2\cosh(c) + (1-\alpha(s)^2)},
\end{align*}
where $e^{2c}= y/x.$ 
We define 
$$f_\alpha(c)= \frac{\cosh(\alpha c)}{\alpha^2\cosh(c) + (1-\alpha^2)}$$
and show that $f_\alpha:[0,\infty)\to (0,1]$ is bijective.
Indeed, notice that  
$$f_\alpha(0)=1 \quad \hbox{and} \quad \lim_{c\to \infty}f_\alpha(c)=0.$$
Continuity and the Intermediate Value Theorem imply that the function $f_\alpha:[0,\infty)\to (0,1]$ is surjective. Moreover, we can show that the function $f_\alpha:[0,\infty)\to (0,1]$ is also injective. To do this, it is enough to show that the function is monotonic on $[0,\infty)$. So, note that 
$$\frac{d}{d c} f_\alpha(c)\le 0$$
if and only if,
$$g_\alpha(c):=\alpha \sinh(\alpha c)(\alpha^2\cosh(c) + (1-\alpha^2))-\alpha^2\sinh(c)\cosh(\alpha c)\le 0.$$
Since, $g_\alpha(0)=0$, it suffices to show that $g_\alpha$ is monotonically decreasing on $[0,\infty)$. Taking a derivative with respect to $c$ we obtain,
$$\frac{d}{d c} g_\alpha (c)=2 \alpha (-1 + \alpha^2) \cosh(c \alpha) \sinh(c/2)^2$$
which is clearly non-positive when $c\ge 0$. Hence, the function $f_\alpha:[0,\infty)\to (0,1]$ is bijective. To obtain a solution for \eqref{Prop4ab}, fix $s\in (0, 1/2)\cup (1/2,1)$ and set $c=f_{\alpha(s)}^{-1} (a/b)$. With this, we can obtain the desired  $x$ and $y$ satisfying \eqref{Prop4ab}.

%Therefore, the existence of $x$ and $y$ satisfying (\ref{Prop4ab}) follows from the surjectivity of $f(\alpha, c)$ onto $(0,1]$ which can be obtained by the Intermediate Value Theorem.  % implies that the function is surjective onto $(0,1]$. This completes the proof.
\end{proof}

\begin{remark}
Using similar arguments one can prove that if one of the following inequalities holds for any non-negative numbers $x \le y$,
\begin{itemize}
\item[1)] $f(x) \le f(\sqrt{xy});$\smallskip
\item[2)] $f\left(\dfrac{x+y}{2}\right) \le f(y)$,
\item[3)] ${\displaystyle
 f\left(\frac{x^s y ^{1-s}+x^{1-s}y^s}{2}\right) \le f\left(|2s-1|\frac{x+y}{2} + (1-|2s-1|)\sqrt{xy} \right).}
$
\end{itemize}
then the function $f$ is increasingly monotone on $\mathbb{R}^+.$
\end{remark}

\section{Characterization of Operator Monotone Functions via The Geometric Mean}\label{section 3}

In this section we use characterizations of symmetric means given in \cite{Audenaert} and of self-adjoint means given in \cite{Hansen} to establish matrix generalizations of the main results in the previous section.   

\subsection{Symmetric Means via Integral Representations}\label{Hansen}

Let us recall the definition of symmetric operator means. 
\begin{definition}
Let $f:\mathbb R_+\to \mathbb R_+$. We say that $f$ is {\it symmetric} (or $f\in \mathcal{F}_{op}$) if it satisfies the following conditions:
\begin{enumerate}
    \item $f$ is operator monotone,
    \item $tf(t^{-1})=f(t)$ for all $t\in \mathbb R_+$, and
    \item $f(1)=1$.
\end{enumerate}
\end{definition}
Notice that functions $\mathcal{F}_{op}$ are in one-to-one correspondence with symmetric means.  In \cite{Audenaert}, Audenaert {\it et.al.} introduced a new order in the set of symmetric functions as follows. 
\begin{definition}
For $f,g\in \mathcal{F}_{op}$, define
$$\psi(t)=\frac{t+1}{2}\frac{f(t)}{g(t)}, \quad t>0.$$
We say $f\preceq g$ if and only if $\psi \in \mathcal{F}_{op}$.
\end{definition}

It is clear that if $f\in \mathcal{F}_{op}$, $\frac{2t}{1+t}\preceq f(t)\preceq \frac{1+t}{2}$ as $\psi(t)=t^{-1}f(t)$ or $\psi(t)=f(t)$ in these particular cases, both of which are operator monotone. It is shown in \cite{Audenaert} that $\mathcal{F}_{op}$ forms a lattice under $\preceq$. It is worth noting that this order is stronger than the regular point-wise order $\le$. That is, if $f\preceq g$ then $f\le g$. 

The condition $f\in \mathcal{F}_{op}$ implies that $f$ has an integral representation of the form
\begin{equation}\label{IntRep}
f(t)=\frac{1+t}{2}e^{H(t)},
\end{equation}
where $$H(t)=\int_0^1 \frac{(\lambda^2-1)(1-t)^2}{(t+\lambda)(1+t\lambda)(\lambda+1)^2}h(\lambda)\, d\lambda$$
and $h:[0,1]\to [0,1]$ is a measurable function that is uniquely determined by $f$ {\it a.e.} \cite[Proposition 2.1]{Audenaert}. They also showed that \cite[Theorem 2.4]{Audenaert} $$f\preceq g \implies h_f\ge h_g \quad \textit{a.e.}$$  If $f\preceq g$ and $h_f \ne h_g$ on a set of non-zero measure, we will say  $f\prec g$.

\begin{lemma}\label{lemmaIncr}
Let $f\in \mathcal{F}_{op}$ and define $$\varphi(t)= t^{-1}f(t^2).$$
Then, 
\begin{enumerate}
\item If  $\sqrt{\cdot }\prec f$ then, as a real function, $\varphi$ is monotonically decreasing on $(0,1)$ and monotonically increasing on $(1,\infty)$.  
\item If  $\sqrt{\cdot }\succ f$ then $\varphi$ is monotonically increasing on $(0,1)$ and monotonically decreasing on  $(1,\infty)$.
\end{enumerate}
\end{lemma}
\begin{proof}
Consider the derivative 
$$\varphi'(t) = -t^{-2}f(t^2)+2f'(t^2).$$
To show monotonicity as a real function, it suffices to show 
$$2tf'(t)\lessgtr f(t)$$
depending on the interval and the order relationship considered. Based on \eqref{IntRep} we consider
$$2tf'(t)=te^{H(t)}(1+(1+t)H'(t))\lessgtr f(t)$$
if and only if
$$H'(t)\lessgtr \frac{1-t}{2t(1+t)}.$$
Explicitly calculating $H'(t)$ we obtain
$$H'(t)=\int_0^1 \left(\frac{1}{(t+\lambda)^2}-\frac{1}{(1+t\lambda)^2}\right)h(\lambda) \, d\lambda =\int_0^1 \frac{(1-\lambda^2)(1-t^2)}{(t+\lambda)^2(1+t\lambda)^2}h(\lambda)\, d\lambda.$$  
An easy calculation shows that when $h(\lambda)$ is substituted by the constant function $1/2$, the integral becomes
$$\frac{1}{2}\int_0^1 \frac{(1-\lambda^2)(1-t^2)}{(t+\lambda)^2(1+t\lambda)^2}\, d\lambda=\frac{1-t}{2t(1+t)}.$$
So now we apply \cite[Theorem 2.4]{Audenaert} to determine the monotonicity of $\varphi$ in each case. So, let $\sqrt{\cdot }\prec f$ and $t\in (0,1)$. In this case $h(\lambda)\le 1/2$ and the integrand,
$$\frac{(1-\lambda^2)(1-t^2)}{(t+\lambda)^2(1+t\lambda)^2}h(\lambda)\ge 0$$
for all $(t,\lambda)\in (0,1)\times [0,1]$. Therefore,
$$H'(t)\le \frac{1-t}{2t(1+t)},$$
which implies that $\varphi$ is monotonically decreasing on $(0,1)$. When $t\in (1,\infty)$ the integrand is non-positive and the inequality is reversed, yielding that $\varphi$ is monotonically increasing on that interval. The analysis for $\sqrt{\cdot }\succ f$ is similar, but in this case $h(\lambda)\ge 1/2$.
\end{proof}

\begin{remark}
Another way to obtain the previous result would be by using the monotonicity on one interval and using the fact that $\varphi(t)=\varphi(t^{-1})$ by the symmetry condition (2) in the definition of the class $\mathcal{F}_{op}$. As a corollary, $\varphi$ has an absolute minimum/maximum at the point $(1,1)$.
\end{remark}

Suppose that $\sqrt{\cdot }\prec f$ . Then, $\sqrt{t}<f(t)$ for some $t\in(1,\infty)$. By the preceding lemma $\varphi$ is monotonically increasing on this interval, so
 $$\gamma:= \lim_{t\to \infty} \varphi(t) = \lim_{t\to \infty}  t^{-1}f(t^2)>1.$$
As a result the interval $(1,\gamma)$ is non-empty. 

On the other hand, suppose that $\sqrt{\cdot }\succ f$. Then, $\sqrt{t}>f(t)$ for some $t\in(1,\infty)$. In this case, however, $\varphi$ is monotonically decreasing on this interval, so
 $$\gamma:= \lim_{t\to \infty} \varphi(t) <1$$
and $(1,\gamma)$ is non-empty. 
 
\begin{lemma}\label{lemma12}
Let $\sigma$ be some symmetric operator mean on $\mathbb{R}^+$ with representing function $f$ such that $\sqrt{\cdot }\prec f$ (resp.  $\sqrt{\cdot }\succ f$) and let $\gamma = \lim_{t\to \infty}f(t^2)/t$. Then, if $X$ and $Y$ are positive definite operators such that $X\le Y < \gamma X$ (resp. $ \gamma X < Y \le X$), then there exist positive operators $A$ and $B$ such that 
\begin{equation*}
X =A\# B \quad \text{ and }\quad  Y=A\sigma B. 
\end{equation*}
\end{lemma}
\begin{proof}
Note that if we show that for $I\le X^{-1/2}YX^{-1/2}:=Y_0\le \gamma I_n$ we can find positive operators $A_0$ and $B_0$ such that:
\begin{equation*}
I_n =A_0\# B_0 \quad \text{ and } \quad Y_0=A_0\sigma B_0,
\end{equation*}
we can obtain the desired result by choosing $A:= X^{1/2}A_0X^{1/2}$ and $B:= X^{1/2}B_0X^{1/2}$. This is equivalent to the following problem: Given $I_n\le Y_0 \le \gamma I_n$ find $A_0\ge 0$ such that $$Y_0=A_0\sigma A_0^{-1}.$$
So, define $\varphi(t):=t\sigma t^{-1}=t f(t^{-2})$. By symmetry, we have that $\varphi(t)=t^{-1}f(t^2)$.  Since $\varphi(t)$ is continuous on $[1,\infty)$ and $\varphi(1)=f(1) = 1$, the function is bijective from $[1, \infty)$ onto $[1,\gamma)$. And so, we can define $A_0=\varphi^{-1}(Y_0)$. This gives the desired result. The proof for the case when $\sqrt{\cdot }\succ f$ is identical, but uses the fact that in this case $\varphi:[1,\infty)\to (\gamma, 1]$ is bijective instead.
\end{proof}

\begin{theorem}\label{Prop13}
Let $\sigma$ be some symmetric operator mean on $\mathbb{R}^+$ with representing function $f$ such that $\sqrt{\cdot }\prec f$. Then, if \begin{equation}\label{Hansen1}
g(A\# B) \le g(A\sigma B) 
\end{equation}
for any positive operators $A$ and $B$, then the function $g$ is operator monotone on $\mathbb{R}^+$. If, on the other hand, $\sqrt{\cdot }\succ f$ and 
\begin{equation}\label{Hansen2}
g(A\# B) \ge g(A\sigma B), 
\end{equation}
then $g$ is operator monotone on $\mathbb{R}^+$.
\end{theorem}
\begin{proof}
First we prove \eqref{Hansen1}. Let $f$ and $\varphi$ be as in the proof of Lemma \ref{lemma12}. Assume that $f\succ \sqrt{\cdot}$ and choose $\gamma_0\in(1,\gamma)$. Let $0<X\le Y$ and $Y_0=X^{-1/2}Y X^{-1/2}$. Consider the spectral decomposition, $Y_0=\sum_{i=1}^n \lambda_i P_i$, with the eigenvalues $\lambda_i$ listed in non-ascending order. Then, there exists a set of non-ascending integers $\{m_i | 1\le i \le n\}$ such that 
$$\gamma_0^{m_i}< \lambda_i \le \gamma_0^{m_i+1}. $$
In particular, we have
\begin{multline*}
    I< \gamma_0 I < \gamma_0^2 I<...< \gamma_0^{m_n}I < \lambda_n P_n + \sum_{i=1}^{n-1}\gamma_0^{m_n}P_i \le  \lambda_n P_n + \sum_{i=1}^{n-1}\gamma_0^{m_n+1}P_i<\\  ... \le \sum_{j=0}^k\lambda_{n-j}P_{n-j}+ \sum_{i=1}^{n-k-1}\gamma_0^{m_{n-k}}P_i\le\sum_{j=0}^k\lambda_{n-j}P_{n-j}+ \sum_{i=1}^{n-k-1}\gamma_0^{m_{n-k}+1}P_i
    \\ \le ... \le  Y_0 \le \gamma_0^{m_1}I.
\end{multline*}
Multiplying each term of the chain of inequalities on both sides by $X^{1/2}$, we obtain the chain inequalities
$$
   0\le X \le \gamma_0 X \le \gamma_0^2 X<...<  Y \le \gamma_0^{m_1}X.
$$
Now consider the $k$-th and $k+1$-st terms of this chain. They satisfy the inequality,
\begin{multline*}
    Z_k:=\sum_{j=0}^k\lambda_{n-j}X^{1/2}P_{n-j}X^{1/2}+ \sum_{i=1}^{n-k-1}\gamma_0^{m_{n-k}}X^{1/2}P_iX^{1/2} \\ \le Z_{k+1}:=\sum_{j=0}^k\lambda_{n-j}X^{1/2}P_{n-j}X^{1/2}+ \sum_{i=1}^{n-k-1}\gamma_0^{m_{n-k}+1}X^{1/2}P_iX^{1/2}
    \\ \le \gamma\left(\sum_{j=0}^k\lambda_{n-j}X^{1/2}P_{n-j}X^{1/2}+ \sum_{i=1}^{n-k-1}\gamma_0^{m_{n-k}}X^{1/2}P_iX^{1/2}\right)=\gamma Z_k.
\end{multline*}
Thus, Lemma \ref{lemma12} implies that there exist  positive operators $A_k$ and $B_k$  such that:
\begin{equation*}
Z_k =A_k\# B_k \quad \text{ and } \quad  Z_{k+1} =A_k\sigma B_k 
\end{equation*}
and so 
$$g(X) \le g(Z_1) \le g(Z_2) \le  ...\le g(Z_n)=g(Y).$$
The proof of \eqref{Hansen2} is similar, hence omitted. 
\end{proof}

\begin{remark}
The proof of the theorem could also be obtained by a similar argument as the proof of \cite[Proposition 4.1]{hiai-ando}. Roughly, this would be achieved by constructing a decreasing sequence of positive numbers $\{a_k\}$ such that $a_k\to 0$ as $k\to \infty$ and $\gamma \ge (1-a_{k+1})/(1-a_k)$ for each $k$. With this sequence, one could construct a sequence of positive matrices with $k$-th term $$X_k=a_k X+(1-a_k)Y.$$
In this case, $0< X \le Y$ implies
$$X_k\le X_{k+1}\le \gamma X_k.$$
By its construction, the sequence $X_k$ converges to $Y$ in the operator norm. So, the continuity of $f$ implies $f(X)\le f(Y)$. Instead, in our proof we use the spectral decomposition of $Y_0$ to provide an explicit construction of a finite set of matrices that gives the result while avoiding the limiting process.
\end{remark}

\subsection{Self-Adjoint Means via Integral Representation}\label{SelfAdjointSection}

A mean $\sigma$ is said to be {\it self-adjoint} if it satisfies $$(A\sigma B)^{-1} = A^{-1}\sigma B^{-1} \quad \hbox{for any}\quad A, B > 0.$$
There exists a one-to-one correspondence between self-adjoint means and the class of operator monotone functions $\mathcal{E}$ defined below. This correspondence was considered by Hansen in \cite{Hansen} and a characterization was given in terms of the exponential of an integral. 

\begin{definition}
Let $f:\mathbb R_+\to \mathbb R_+$. We say that $f\in \mathcal{E}$ if it satisfies the following conditions:
\begin{enumerate}
    \item $f$ is operator monotone, and
    \item $f(t^{-1})=f(t)^{-1}$ for all $t\in \mathbb R_+$.
\end{enumerate}
\end{definition}
The aforementioned characterization is proved in \cite[Theorem 1.1]{Hansen} and it states that 
$$f(t)=\exp \int_{-1}^0\left(\frac{1}{\lambda-t}+\frac{t}{1-\lambda t}\right)h(\lambda)\, d\lambda,$$
where $h:[-1,0]\to[0,1]$ is a measurable function whose class is uniquely determined by $f$.

\begin{definition}
For $f,g\in \mathcal{E}$, we say $f\succeq_{sa} g$ if and only if $fg^{-1}$ is operator monotone.
\end{definition}

In the following, we show that this so defined relation satisfies the same properties as the order defined in \cite{Audenaert} on $\mathcal{F}_{op}$ that we introduced earlier in this section. 

\begin{proposition}\label{SAOrder}
Let $f,g\in \mathcal{E}$. Then, $f\succeq_{sa} g$ if and only if $h_f\ge h_g$ {\it a.e.}
\end{proposition}
\begin{proof}
Note that $f,g \in \mathcal{E}$ implies that $(f/g)(t^{-1})=((f/g)(t))^{-1}.$
So, requiring $fg^{-1}$ be operator monotone is equivalent to requiring $fg^{-1}\in \mathcal{E}$. Therefore, there exists a class of measurable functions $h_{fg^{-1}}:[-1,0]\to [0,1]$ such that 
$$(fg^{-1})(t)=\exp \int_{-1}^0\left(\frac{1}{\lambda-t}+\frac{t}{1-\lambda t}\right)h_{fg^{-1}}(\lambda)\, d\lambda,$$
and 
$$h_{fg^{-1}}(\lambda)=h_f(\lambda)-h_g(\lambda) \quad {\it a.e.}$$
The result follows from this observation.
\end{proof}

\begin{proposition}
The set $\mathcal{E}$ together with the order relation $\succeq_{sa}$ is a lattice with minimal and maximal elements $f(t)=1$ and $f(t)=t$, respectively. Furthermore, there exists an involutive order reversing operation $\dagger:\mathcal{E}\to \mathcal{E}$.
\end{proposition}
\begin{proof}
 From Proposition \ref{SAOrder}, it is clear that $\succeq_{sa}$ defines an order relation on $\mathcal{E}$. Moreover, it is easy to see that $f\in \mathcal{E}$ implies that $1\preceq_{sa} f(t)\preceq_{sa} t$. Indeed, $1\preceq_{sa} f(t)$ follows from the monotonicity of $f$ and  $f(t)\preceq_{sa} t$ follows from the monotonicity of $\frac{t}{f(t)}$.

We  define the meet and join of any two elements in a similar fashion as in \cite{Audenaert}. For $f,g\in \mathcal{E}$  define:
\begin{align*}
    f\wedge g & =  \exp \int_{-1}^0\left(\frac{1}{\lambda-t}+\frac{t}{1-\lambda t}\right)\min\{h_f(\lambda),h_g(\lambda)\} \, d\lambda,\\
    f \vee g  & =   \exp \int_{-1}^0\left(\frac{1}{\lambda-t}+\frac{t}{1-\lambda t}\right)\max\{h_f(\lambda),h_g(\lambda)\} \, d\lambda.
\end{align*}
It is easy to see that,
$$ f\wedge g \preceq_{sa} f \preceq_{sa} f\vee g \preceq_{sa}. $$

We now shoe that the map $$f(t)\to f^\dagger(t)=\frac{t}{f(t)}$$
is an involutive order reversing map on $\mathcal{E}$. Indeed, it is easy to see $f^{\dagger\dagger}=f$,  
$$f^\dagger(t^{-1})=\frac{1}{tf(t^{-1})}=\frac{f(t)}{t}=(f^\dagger(t))^{-1},$$
and
$$f\preceq_{sa}g \quad \implies \quad g^\dagger \preceq_{sa} f^\dagger. $$
\end{proof}
Now we turn into a characterization of operator monotone functions using self-adjoint means. As before, if $f\preceq_{sa} g$ and $h_f \ne h_g$ on a set of non-zero measure, we will say  $f\prec_{sa} g$.

\begin{lemma}
Let $f\in \mathcal{E}$ and define $$\varphi(t)= t^{-1}f(t^2).$$
Then, 
\begin{enumerate}
\item If  $\sqrt{\cdot }\preceq_{sa} f$ then, as a real function, $\varphi$ is monotonically increasing on $\mathbb R_+$.  
\item If  $\sqrt{\cdot }\succeq_{sa} f$ then $\varphi$ is  monotonically decreasing on  $\mathbb R_+$.
\end{enumerate}
\end{lemma}
\begin{proof}
As before, to show monotonicity as a real function, it suffices to show 
\begin{equation}\label{LGIneq}2tf'(t)\lessgtr f(t)\end{equation}
depending on the interval and the order relationship considered.  With this expression, \eqref{LGIneq} becomes:
$$2t\int_{-1}^0\left(\frac{1}{(\lambda-t)^2}+\frac{1}{(1-\lambda t)^2}\right)h(\lambda)\, d\lambda \lessgtr 1.$$
The result now follows from the fact that the integrand is non-negative and for $h(\lambda)=1/2$, $f(t)=\sqrt{t}$ and 
$$\int_{-1}^0\left(\frac{1}{(\lambda-t)^2}+\frac{1}{(1-\lambda t)^2}\right)\, d\lambda = \frac{1}{t}.$$
\end{proof}
Using the same arguments as in Lemma \ref{lemma12} and Theorem \ref{Prop13} we can show the following result.
\begin{theorem}\label{Prop11}
Let $\sigma$ be some self-adjoint operator mean on $\mathbb{R}^+$ with representing function $f$ such that $\sqrt{\cdot }\prec_{sa} f$. Then, if  \begin{equation}\label{geomA1}
g(A\# B) \le g(A\sigma B)
\end{equation}
for any positive operators $A$ and $B$ such that $A<B$, then the function $g$ is operator monotone on $\mathbb{R}^+$. If, on the other hand, $\sqrt{\cdot}\prec_{sa} f^{\dagger}$ and 
\begin{equation}\label{geomA2}
g(A\# B) \ge g(A\sigma B), 
\end{equation}
for such positive operators, then $g$ is operator monotone on $\mathbb{R}^+$.
\end{theorem}

\subsection{Characterization with Kubo-Ando Condition}\label{KAConditionSect}

There is yet another class of means to consider. Let $\tau$ and $\tau^\perp$ be the means represented by operator monotone functions $g$ and $g^\dagger$, respectively. Kubo and Ando showed in \cite[Theorem 5.4]{Kubo-Ando} that if an operator mean $\sigma$ with representing function $f$ satisfies 
\begin{equation}\label{KACondition}
(A\tau B)\sigma (A\tau^\perp B)\le A\sigma B
\end{equation}
for a non-trivial mean $\sigma$ and all positive operators $A$ and $B$ then $f\ge \sqrt{\cdot}$. Moreover, in \cite[Theorem 5.7]{Kubo-Ando}, they showed that whenever $\sigma$ satisfies \eqref{KACondition} for every operator mean $\tau$ its representing function $f$ satisfies $t^{-1}f(t^2)$ is non-increasing on $(0,1)$ and non-decreasing on $(1,\infty)$. Moreover, in subsequent corollaries, they showed that if the inequality \eqref{KACondition} is reversed then $f\le \sqrt{\cdot}$ and $t^{-1}f(t^2)$ is non-decreasing on $(0,1)$ and non-increasing on $(1,\infty)$.

These are precisely the behaviors needed in the proof of Lemma \ref{lemma12} and consequently Theorem \ref{Prop13}. Therefore, this allows us to follow the same arguments to show a similar result for this particular class of means.
\begin{theorem}\label{Prop12}
Let $A$ and $B$ be positive operators and $\sigma$ be an operator mean on $\mathbb{R}^+$ satisfying \eqref{KACondition} for every operator mean $\tau$. Assume further that the representing function $f$ satisfies $f(x)>\sqrt{x}$ for some $x\in (1,\infty)$. Then, if  \begin{equation}
g(A\#B) \le g(A\sigma B),
\end{equation}
then the function $g$ is operator monotone on $\mathbb{R}^+$. If, on the other hand, the reversed inequality is satisfied in \eqref{KACondition} for every operator mean $\sigma$, $f(x)<\sqrt{x}$ for some $x\in (0,1)$, and 
\begin{equation}
g(A\#B) \ge g(A\sigma B), 
\end{equation}
then $g$ is operator monotone on $\mathbb{R}^+$.
\end{theorem}

\subsection{General Symmetric Means}

 In this section, we show that the answer to Question (C) is positive in general.  To prove Theorem \ref{Prop13} we used monotonicity of the function $\varphi$ on certain intervals to obtain bijectivity, thus obtaining a well-defined $\varphi^{-1}$ when restricted to the appropriate intervals. With this function, we were able to solve the problem in Lemma \ref{lemma12}, which then allowed us to obtain the desired characterization. With a little care, it is possible to obtain the same result when $\varphi$ is only surjective on the prescribed intervals. 

We recall some of our notation from Section \ref{Hansen}. Suppose that $\sqrt{\cdot }\lessgtr f$, as before define $\varphi(t)=t^{-1}f(t^2)$. Then, we have 
 $$\gamma:= \lim_{t\to \infty} \varphi(t) = \lim_{t\to \infty}  t^{-1}f(t^2)\lessgtr 1.$$

With this we can show a lemma equivalent to Lemma \ref{lemma12}.
 
\begin{lemma}\label{lemma13}
Let $\sigma$ be some symmetric operator mean on $\mathbb{R}^+$ with representing function $f$ such that $\sqrt{\cdot }< f$ (resp.  $\sqrt{\cdot }> f$) and let $\gamma = \lim_{t\to \infty}f(t^2)/t$. Then, if $X$ and $Y$ are positive definite matrices such that $X\le Y < \gamma X$ (resp. $ \gamma X < Y \le X$), then there exist positive matrices $A$ and $B$ such that 
\begin{equation*}
X =A\# B \quad \text{ and } \quad Y=A\sigma B. 
\end{equation*}
\end{lemma}
\begin{proof}
As in the proof of Lemma \ref{lemma12}, we show the lemma when $\sqrt{\cdot }<f$. In this case,  it suffices to show that given $I_n\le Y_0=U\, \text{diag}(\{\lambda_i(Y_0)\})\, U^* \le \gamma I_n$, we can find $A_0\ge 0$ such that $$Y_0=A_0\sigma A_0^{-1}=\varphi(A_0).$$
While $\varphi(t)$ is not necessarily bijective in this case, it is continuous on $[1,\infty)$ and $\varphi(1)=f(1) = 1$. Therefore, the restriction of $\varphi$ to some subset of $[1, \infty)$ is surjective onto $[1,\gamma)$. 

Since $\sigma(Y_0)\subset [1,\gamma)$, surjectivity of the restriction of $\varphi$ implies that the set $$\varphi^{-1}(\lambda_i(Y_0)):=\{x \in [1,\infty) \ | \ \varphi(x)=\lambda_i(Y_0)\} \ne \emptyset.$$
In particular, if we choose $\delta_i(Y_0) \in \varphi^{-1}(\lambda_i(Y_0))$ for each $i$,  the matrix
$$A_0:=U\text{diag}(\{\delta_i(Y_0)\})U^*$$
satisfies 
$$\varphi(A_0)=U\text{diag}(\{\varphi(\delta_i(Y_0))\})U^*=U\text{diag}(\{\lambda_i(Y_0)\})U^*,$$
and the result follows as in Lemma \ref{lemma12}.
\end{proof}

\begin{remark}
In \cite{hiai-ando, hoa}, considering similar questions, the authors required the functions equivalent to $\varphi$ to be bijective. In the proof of this lemma we relax the condition on $\varphi$ to only require surjectivity. This means that the solutions to the desired system of equations may not be unique, but their existence is guaranteed.  
\end{remark}

Now using the same argument as in the proof of Theorem \ref{Prop13} we can show the following theorem.

\begin{theorem}\label{Prop14}
Let $f$ be a continuous function on $(0,\infty)$. Then, $f$ is operator monotone if and only if either one of the following holds:
\begin{enumerate}
    \item If $f(A\#B)\le f(A\sigma B)$ for all positive definite $A$ and $B$ and some symmetric operator mean $ \# < \sigma$.
    \item If $f(A\#B)\ge f(A\sigma B)$ for all positive definite $A$ and $B$ and some symmetric operator mean $ \# >\sigma$.
\end{enumerate}
\end{theorem}

\section{Further Characterizations}\label{OperVers}
Notice that from Equation \eqref{general11} we have the following inequalities for matrix means:
\begin{equation*}
A\# B \le \frac{A\sharp_s B + A\sharp_{1-s} B}{2} \le \alpha(s)^2 \frac{A+B}{2} + (1-\alpha(s)^2) A\sharp B  \le\frac{A+B}{2},
\end{equation*}
In this section, using above inequalities we establish new characterizations of operator monotone functions.

\begin{theorem}\label{theorem 5}
Let $f$ be a continuous function on $[0,\infty)$, $s\in (0, 1/2)\cup (1/2,1)$ and $\alpha = 1-2s$. The following statements are equivalent:
\begin{itemize}
\item[(i)] $f$ is operator monotone on $[0,\infty);$
\item[(ii)] For any positive definite matrices $A$ and $B,$
\begin{equation}\label{jose}
   f(A\sharp B)\le  f\left( \frac{A\sharp_s B + A\sharp_{1-s} B}{2}\right) .
\end{equation}
\item[(iii)] For any positive definite matrices $A$ and $B,$
\begin{equation}\label{bh}
    f\left( \frac{A\sharp_s B + A\sharp_{1-s} B}{2}\right) \le  f\left(\alpha(s)^2 \frac{A+B}{2} + (1-\alpha(s)^2) A\sharp B \right);
\end{equation}

\item[(iv)] For any positive definite matrices $A$ and $B$,
\begin{equation*}
f\left(\alpha(s)^2 \frac{A+B}{2} + (1-\alpha(s)^2) A\sharp B \right) \le f\left(\frac{A+B}{2}\right).
\end{equation*}
\end{itemize}
\end{theorem}
\begin{proof}
It is obvious that (i) implies (ii), (iii), and (iv). Let us show that (iii) implies (i) first and then we show (ii) implies (i). That would complete the proof since (iv) implies (i) follows from \cite[Proposition 4.1]{hiai-ando} both the Heron mean is symmetric.

Suppose (\ref{bh}) holds for any positive definite matrices $A$ and $B$. We need to show that for any $0<  X \le Y$, 
$$
f(X) \le f(Y).
$$
Firstly, let us consider the case when $Y= I_n$, where $I_n$ is the identity matrix or order $n$. We now show that there exist positive definite matrices $A_0, B_0$ such that 
\begin{equation}\label{111}
\frac{A_0\sharp_s B_0 + A_0\sharp_{1-s} B_0}{2} = A_0^{1/2} \left(\frac{C_0^s + C_0^{1-s}}{2}\right)A_0^{1/2} = X\quad \end{equation}
and
\begin{equation}\label{112}
A_0^{1/2}  \left(\alpha(s)^2 \frac{I_n+C_0}{2} + (1-\alpha(s)^2) C_0^{1/2}\right)A_0^{1/2}  = I_n,
\end{equation}
where $C_0 = A_0^{-1/2}B_0A_0^{-1/2}.$ 
From \eqref{112}, we get 
$$
A_0^{1/2} =  \left(\alpha(s)^2 \frac{I_n+C_0}{2} + (1-\alpha(s)^2) C_0^{1/2}\right)^{-1/2}.
$$
Substituting the last identity to \eqref{111}, we get
\begin{equation}\label{113}
X = \left(\frac{C_0^s + C_0^{1-s}}{2}\right).\left(\alpha(s)^2 \frac{I_n+C_0}{2} + (1-\alpha(s)^2) C_0^{1/2}\right)^{-1}
\end{equation}
From the proof of Theorem \ref{Prop4} the function 
$$f(x)=\left(\frac{x^s+x^{1-s}}{2}\right)\left(\alpha(s)^2\frac{1+x}{2}+(1-\alpha(s)^2)\sqrt{x}\right)^{-1}$$
is bijective and takes values in $(0, 1]$. Therefore, for any $0<X \le I_n$ there exists a unique matrix $C_0$ satisfying \eqref{113}.  Hence, the matrix $A_0$ is obtained from \eqref{112} and the matrix $B_0$ equals $A_0^{1/2}C_0A_0^{1/2}$. 

In general, for $0 <X \le Y$  we have $0 < Y^{-1/2}XY^{-1/2} \le I_n$. By the above arguments, we can find $A_0, B_0 \in \mathbb{M}_n^+$ such that  
\begin{equation*}
\frac{A_0\sharp_s B_0 + A_0\sharp_{1-s} B_0}{2} = Y^{-1/2}XY^{-1/2} \end{equation*}
and
\begin{equation*}
\alpha(s)^2 \frac{A_0+B_0}{2} + (1-\alpha(s)^2) A_0\sharp B_0  = I_n.
\end{equation*}
Consequently, applying (\ref{bh}) to matrices $A= Y^{1/2}A_0Y^{1/2}, B=Y^{1/2}B_0Y^{1/2}$ we obtain that $f(X) \le f(Y)$. Finally, by the continuity of $f$ we conclude that the function $f$ is operator monotone on $[0, \infty).$ 

To show that (ii) implies (i), following the same argument, it suffices to show that the function 
$k_s(x):(0,1]\to (0,1]$ defined by $$k_s(x)=\frac{2\sqrt{x}}{x^s+x^{1-s}}$$
is bijective. However, by realizing $k_s$ as a hyperbolic secant, this is obvious.
\end{proof}

\begin{remark} A couple of remarks are in order:
\begin{enumerate}
    \item In Theorem \ref{theorem 5}, $\alpha(s)^2$ can be replaced with $|\alpha(s)|$ and the same result holds.
    \item Since the Heinz mean is a symmetric mean, we have proved a partial answer for Question (C) for the case when $\sigma$ is the Heinz mean.
\end{enumerate}
\end{remark}

{\bf Acknowledgements}

Research of the first author is funded by Vietnam National Foundation for Science and Technology Development (NAFOSTED) under grant number 101.02-2017.310.

%The authors want to thank Professor Fumio Hiai the reviewers for spotting an error in the original manuscript and for their suggestions which improved the quality of the present note.

\end{document}